\documentclass[a4paper]{scrartcl}
\usepackage{amssymb, amsmath, amsthm, mathrsfs, pxfonts, braket}

\usepackage[T1]{fontenc}  
\usepackage{hyperref}     

\usepackage[all,2cell]{xy}
\objectmargin+{1mm}
\labelmargin+{0.8mm}
\SelectTips{cm}{12}

\usepackage{lmodern}

\usepackage{enumitem}
\setlist[enumerate]{itemsep=0pt,label=$(\mathrm{\roman*})$, topsep=5pt}
\setlist[itemize]{itemsep=0pt, topsep=5pt, 
labelindent=\parindent,leftmargin=*}

\newtheorem{thm}{Theorem}[section]
\newtheorem{cor}[thm]{Corollary}
\newtheorem{lem}[thm]{Lemma}
\newtheorem{prop}[thm]{Proposition}

\theoremstyle{definition}
\newtheorem{dfn}[thm]{Definition}

\newtheorem*{claim}{Claim}

\usepackage{etoolbox}
\usepackage{framed}
\newcommand\enclosebox[2]{%
  \BeforeBeginEnvironment{#1}{\begin{#2}}%
  \AfterEndEnvironment{#1}{\end{#2}}%
}

\enclosebox{thm}{leftbar}
\enclosebox{cor}{leftbar}
\enclosebox{lem}{leftbar}
\enclosebox{prop}{leftbar}
\enclosebox{dfn}{leftbar}
\enclosebox{rem}{leftbar}
\enclosebox{ex}{leftbar}

\newcommand{\ab}{\mathrm{ab}}
\newcommand{\Char}{\operatorname{char}}

\newcommand{\Ct}{C^{\mathrm{t}} }
\newcommand{\CtX}{\Ct(X)}
\newcommand{\CtXg}{\CtX^0}

\newcommand{\CXbar}{C(\Xbar)}
\newcommand{\CXbarg}{\CXbar^0}
\newcommand{\CXbargv}{(\CXbarg)^{\vee}}

\newcommand{\CXbarv}{\CXbar^{\vee}}

\newcommand{\CXpbar}{C(\Xbar')}

\newcommand{\CXp}{C(X')}

\newcommand{\CXv}{\CX^{\vee}}
\newcommand{\CX}{C(X)}

\newcommand{\CXD}{C(X,D)}
\newcommand{\CXDg}{\CXD^0}

\newcommand{\CY}{C(Y)}

\newcommand{\Cf}{\textit{cf.}\;}
\newcommand{\Coker}{\operatorname{Coker}}
\newcommand{\Cokertop}{\Coker_{\mathrm{top}}}
\renewcommand{\d}{\partial}

\newcommand{\dx}{\d_x}

\newcommand{\fbar}{\ol{f}}

\newcommand{\Fbar}{\ol{F}}
\newcommand{\fil}{\operatorname{fil}}

\newcommand{\Gal}{\operatorname{Gal}}

\newcommand{\HGal}{H_{\Gal}}
\newcommand{\Hom}{\operatorname{Hom}}

\newcommand{\isomto}{\stackrel{\simeq}{\to}}
\renewcommand{\Im}{\operatorname{Im}}
\newcommand{\inj}{\hookrightarrow}
\newcommand{\ilim}{\varinjlim}

\newcommand{\Id}{\operatorname{id}}

\newcommand{\Ker}{\operatorname{Ker}}
\newcommand{\Kt}{K^{\times}}

\newcommand{\kx}{k(x)}
\newcommand{\kxt}{\kx^{\times}}

\newcommand{\kX}{k(X)}
\newcommand{\kXx}{k(X)_x}

\newcommand{\ky}{k(y)}
\newcommand{\kyt}{\ky^{\times}}
\newcommand{\kY}{k(Y)}
\newcommand{\kYy}{\kY_y}
\newcommand{\kbar}{\ol{k}}
\newcommand{\Kbar}{\ol{K}}

\renewcommand{\L}{\mathbb{L}}

\newcommand{\m}{\mathfrak{m}}
\newcommand{\mK}{\m_K}
\newcommand{\NL}{\N(\L)}
\newcommand{\N}{\mathbb{N}}
\newcommand{\Np}{\N'}

\newcommand{\ol}[1]{\overline{#1}}
\renewcommand{\O}{\mathscr{O}}
\newcommand{\OK}{O_K}
\newcommand{\OKt}{\OK^{\times}}

\newcommand{\opcit}{\textit{op.\,cit.}}
\newcommand{\onto}[1]{\stackrel{#1}{\to}}

\newcommand{\plim}{\varprojlim}
\newcommand{\piab}{\pi_1^{\ab}}
\newcommand{\pitab}{\pi_1^{\mathrm{t},\ab}}

\newcommand{\piabX}{\piab(X)}

\newcommand{\pitabX}{\pitab(X)}
\newcommand{\pitabXg}{\pitab(X)^0}

\newcommand{\piabXbar}{\piab(\Xbar)}
\newcommand{\piabXbarg}{\piab(\Xbar)^0}
\newcommand{\piabXbargv}{(\piab(\Xbar)^0)^{\vee}}

\newcommand{\piabXD}{\piab(X,D)}

\newcommand{\piabY}{\piab(Y)}

\newcommand{\Q}{\mathbb{Q}}
\newcommand{\QZ}{\Q/\Z}

\newcommand{\Res}{\operatorname{Res}}

\newcommand{\rhoXD}{\rho_{X,D}}

\newcommand{\rhoX}{\rho_X}

\newcommand{\rhoXv}{\rhoX^{\vee}}
\newcommand{\rhoXm}{\rho_{X,m}}
\newcommand{\rhoXmv}{\rhoXm^{\vee}}

\newcommand{\rhoXt}{\rho_X^{\mathrm{t}}}
\newcommand{\rhoXtg}{\rho_X^{\mathrm{t},0}}

\newcommand{\rhoXbar}{\rho_{\Xbar}}
\newcommand{\rhoXbarg}{\rho_{\Xbar}^0}
\newcommand{\rhoXbargv}{(\rho_{\Xbar}^0)^{\vee}}

\newcommand{\rhoXbarv}{\rhoXbar^{\vee}}
\newcommand{\rhoXbarm}{\rho_{\Xbar,m}}
\newcommand{\rhoXbarmv}{\rhoXbarm^{\vee}}

\newcommand{\rXbar}{r(\Xbar)}

\DeclareMathOperator*{\restprod}%
 {\mathchoice{\ooalign{\ensuremath{\displaystyle\prod}\crcr\ensuremath{\displaystyle\coprod}}}%
             {\ooalign{\ensuremath{\textstyle\prod}\crcr\ensuremath{\textstyle\coprod}}}%
             {\ooalign{\ensuremath{\scriptstyle\prod}\crcr\ensuremath{\scriptstyle\coprod}}}%
             {\ooalign{\ensuremath{\scriptscriptstyle\prod}\crcr\ensuremath{\scriptscriptstyle\coprod}}}%
 }

\newcommand{\surj}{\twoheadrightarrow}
\newcommand{\ssm}{\smallsetminus}
\newcommand{\Spec}{\operatorname{Spec}}

\newcommand{\tor}{\mathrm{tor}}
\newcommand{\ur}{\mathrm{ur}}
\newcommand{\Ver}{\operatorname{Ver}}
\newcommand{\vp}{\varphi}
\newcommand{\vK}{v_K}

\newcommand{\wh}[1]{\widehat{#1}}

\newcommand{\Xbar}{\ol{X}}
\newcommand{\Xinf}{X_{\infty}}
\newcommand{\Xpbar}{\Xbar'}

\newcommand{\Ybar}{\ol{Y}}

\newcommand{\Yinf}{Y_{\infty}}

\newcommand{\Z}{\mathbb{Z}}
\newcommand{\Zhat}{\wh{\Z}}

\newcommand{\sn}{\smallskip\noindent}

\title{Class field theory for open curves over local fields}
\author{Toshiro Hiranouchi}

\begin{document}
\pagenumbering{arabic}
\maketitle

\begin{abstract}
We study the class field theory for open curves over a local field. 
After introducing the reciprocity map,  
we determine the kernel and the cokernel of this map. 
In addition to this, 
the Pontrjagin dual of the reciprocity map is also investigated.  
This gives the one to one correspondence 
between the set of finite abelian \'etale 
coverings and the set of finite index open 
subgroups of the id\`ele class group as in the classical class field theory 
under some assumptions.
\end{abstract}


\section{Introduction}
\label{Introduction}
In this note, 
we present the class field theory 
for open (=non proper) curves over a local field with 
arbitrary characteristic. 
Here, a \textbf{local field} means a complete discrete valuation field 
with finite residue field. 
For 
a local field with characteristic $0$,   
a large number of studies have been made 
even for higher dimensional open varieties over local fields 
(e.g., \cite{JS03}, \cite{Hir10}, \cite{Yam09}, and \cite{Yam13}).  
Accordingly, 
our main interest is in the case of positive characteristic local fields. 

To state our results precisely, 
let $k$ be a local field with $\Char(k) = p>0$. 
Let $\Xbar$ be a proper, smooth and geometrically connected curve over $k$ and 
$X$ a nonempty open subscheme in $\Xbar$.  
We often say that the pair 
$X \subset \Xbar$ is an \textbf{open curve} (\Cf Def.\ \ref{def:open_curve}).
A topological group  
$\CX$  which is called the {id\`ele class group}, 
and the {reciprocity map} $\rhoX:\CX \to \piabX$ 
are introduced 
as in \cite{Hir10} (Def.\ \ref{def:CX} and Def.\ \ref{def:rhoX}). 
In this note, we determine the kernel $\Ker(\rhoX)$ 
and the cokernel  
$\Cokertop(\rhoX) := \piabX/\ol{\Im(\rhoX)}$
of $\rhoX$ as (Hausdorff) topological groups, 
where $\ol{\Im(\rhoX)}$ is the topological closure  
of the image $\Im(\rhoX)$. 
One of the main results in this note is the following theorem. 

\begin{thm}[Thm.\ \ref{thm:coker} and Thm.\ \ref{thm:main2}]
\label{thm:main2intro}
Let $X \subset \Xbar$ be as above. 
For the reciprocity map $\rhoX:\CX\to\piabX$, we have 

\begin{enumerate}
\item 
$\Cokertop(\rho_X) \simeq \Zhat^r$ for some $r \in \Z_{\ge 0}$, and   
\item $\Ker(\rhoX)$ is the maximal $l$-divisible subgroup of $\CX$ for all prime number $l\neq p$. 
\end{enumerate}
\end{thm}

\noindent
The theorem above is known 
for $X = \Xbar$(\cite{Sai85}, \cite{Yos03}) 
which corresponds to the unramified class field theory.   
Here, 
the invariant $r = \rXbar$ 
called the \textbf{rank} of $\Xbar$ (\cite{Sai85}, Def.\ 2.5) 
which 
depends on the type of the reduction of $\Xbar$. 
It is known that 
the quotient group 
$\Cokertop(\rho_X) = \piabX/\ol{\Im(\rhoX)}$
classifies \textbf{completely split coverings} of $X$, 
that is, 
finite abelian \'etale coverings of $X$ 
in which  
any closed point $x \in X$ splits completely 
(\cite{Sai85}, Chap.\ II, Def.\ 2.1). 
For example, we have $r=0$ if $\Xbar$ has good reduction 
(see also Thm.\ \ref{thm:ucft} and its remark). 

For a topological abelian group $G$, 
we define the \textbf{Pontrjagin dual group} by  
\[
\xymatrix@C=5mm{
G^{\vee}:= \set{ \mbox{continuous homomorphism $G\to \QZ$ with finite order}}
}
\] 
(\Cf Notation). 
Using this, the reciprocity map $\rhoX$ induces $\rhoXv : \piabX^{\vee} \to  \CXv$.

\begin{thm}[Thm.\ \ref{thm:main}]
\label{thm:mainintro} 
Let $X \subset \Xbar$ be as above.  
We assume $r(\Xbar) = 0$. 
Then, the map $\rhoXv:\piabX^{\vee} \to \CXv$ is bijective. 
\end{thm}

%
%
Since $\piabX$ is compact,   
the injectivity of $\rhoXv$ in Thm.\ \ref{thm:mainintro} 
is deduced from Thm.\ \ref{thm:main2intro} (i).  
However, our id\`ele class group $\CX$ may not be locally compact.  
We have to determine $\Ker(\rhoX)$ and $\Coker(\rhoXv)$ 
independently. 
From the second main theorem (Thm.\ \ref{thm:mainintro}), under the assumption $r(\Xbar) = 0$,  
we have the following one to one correspondence as in the classical class field theory: 
\[
\xymatrix@C=5mm{
 \Set{\mbox{finite abelian \'etale covering of $X$}} 
 \ar@{<->}[r]^-{\mathrm{1:1}} & \Set{\mbox{finite index 
 open subgroup of $C(X)$}}. 
}
\]

\subsection*{Contents} 
The contents of this note is the following:

\begin{itemize}
\item Sect.\ \ref{sec:local}\,: We review some definitions and results of class field theory 
for $2$-dimensional local fields following \cite{Kat79} and \cite{Kat80}. 
 \item Sect.\ \ref{sec:class groups}\,: 
	For an open curve $X \subset \Xbar$ over a local field, 
	the id\`ele class group $\CX$ and the reciprocity map 
	$\rhoX:\CX \to \piabX$ are introduced (Def.\ \ref{def:CX} and Def.\ \ref{def:rhoX}). 
	We also define the fundamental group $\piabXD$ as a quotient of $\piabX$ 
	which classifies abelian \'etale coverings of $X$ with bounded ramification along a
	given effective Weil divisor $D$ on $\Xbar$ (Def.\ \ref{def:piabXD}). 

\item Sect.\ \ref{sec:main}\,: After recalling the unramified class field theory (Thm.\ \ref{thm:ucft}),  
	we study the structure of the tame fundamental group $\pitabX = \piab(X, \Xinf)$, 
	where $\Xinf = \sum_{x\in \Xbar\ssm X} 1 [x]$ considering as a Weil divisor on $\Xbar$.  
	Using this structure theorem, we prove Thm.\ \ref{thm:main2intro} (=Thm.\ \ref{thm:main2}).

\item Sect.\ \ref{sec:mainv}\,: 
	Following the proof of the class field theory for 
	curves over {\it global fields} (\cite{KS83a}, Thm.\ 3, \cite{KS86}, Thm.\ 9.1) 
	basically, 
	we show Thm.\ \ref{thm:mainintro} (=Thm.\ \ref{thm:main}). 
\end{itemize}

\subsection*{Notation}
In this note,  
a {\bf local field} we mean 
a complete discrete valuation field with finite residue field. 
Throughout this note, 
we use the following notation: 
\begin{itemize}
\item $p$\,: a fixed prime number, and 
\item $\Np$\,: the set of $m\in \Z_{\ge 1}$ which is prime to $p$. 
\end{itemize}

\noindent 
For a field $F$, 
\begin{itemize}
\item $\Char(F)$\,: the characteristic of $F$, 
\item $\Fbar$\,: a separable closure of $F$,  
\item $G_F:=\Gal(\Fbar/F)$\,: the Galois group of the extension $\Fbar/F$, 
\item $F^{\ab}$\,: the maximal abelian extension of $F$ in $\Fbar$, 
\item $G_F^{\ab}:= \Gal(F^{\ab}/F)$\,: the Galois group of $F^{\ab}/F$, 
\item $H^n_{\Gal}(F,M)$\,: the Galois cohomology group 
of $G_F$ with coefficients in a $G_F$-module $M$ (\Cf \cite{Kat82b}), 
and 
\item $K_2(F)$\,: the Milnor $K$-group of degree $2$ 
which is defined by 
\[
  K_2(F) = \left(F^{\times} \otimes_{\Z} F^{\times}\right)/J
\]
where $J$  is the subgroup generated by 
elements of the form $a\otimes (1-a)$ $(a \in F^{\times})$.  
The element in $K_2(F)$ represented by $a\otimes b \in F^{\times}\otimes_{\Z} F^{\times}$  
is denoted by $\set{a,b}$ (\Cf \cite{Mil70}). 
\end{itemize}

\sn 
Let $A$ be an abelian group whose operation is written additively.   
The abelian group $A$ is said to be 
\textbf{divisible} if, for every $n\in \Z_{\ge 1}$ and every $x \in A$, 
there exists $y \in A$ such that $ny=x$.  
The abelian group $A$ is \textbf{$l$-divisible} for a prime $l$, 
if for all  $n\in \Z_{\ge 1}$ and every $x \in A$, there exists $y \in A$ such that $l^ny=x$.
For $n\in \Z_{\ge 1}$, 
we use the following notation on $A$: 
\begin{itemize}
\item $A/n:=$ the cokernel of the map $n : A \to A$ 
defined by $x\mapsto nx$, and 
\item $A_{\tor}$\,: the torsion part of $A$. 
\end{itemize}
When $A$ is a topological abelian group, define 
\begin{itemize} 
\item $A^{\vee}$\,: the set of all continuous homomorphisms $A\to \QZ$ 
of {\it finite order}, 
where $\QZ$ is given the discrete topology. 
\end{itemize}
A {\bf curve} over a field $F$ means  
an integral separated scheme of dimension $1$ over $\Spec(F)$. 
For a connected Noetherian scheme $X$, we denote by 
\begin{itemize}
\item $\piabX$\,: the abelianlization of the \'etale fundamental group of $X$ 
(\cite{SGA1}) omitting the base point,  
\item $H^n(X,\mathscr{F})$\,: the \'etale cohomology group of an \'etale sheaf 
$\mathscr{F}$ on $X$, and 
\item $H^n_Z(X,\mathscr{F})$\,: the \'etale cohomology group of an \'etale sheaf 
$\mathscr{F}$ on $X$ with support in $Z$.
\end{itemize}

%

\subsection*{Acknowledgments}
This work was supported by KAKENHI 25800019.

\section{Local class field theory}
\label{sec:local}

For a field $F$ with $\Char(F) = p$, 
$n \in \Z_{\ge 0}$ and $r\in \Z_{\ge 1}$, 
we define 
\[
	H^q_{\Gal}(F, \Z/p^r(n)) := H^{q-n}_{\Gal}(F, 
	W_r\Omega_{\Fbar,\log}^n), 
\]
where $W_r\Omega_{\Fbar,\log}^n$ 
is the Galois module 
defined by the 
\'etale sheaf of 
the logarithmic part of 
the de Rham-Witt complex (\cite{Ill79}).  
Recall that $\Np$ is 
the set of $m\in\Z_{\ge 1}$ which is prime to $p$ (\Cf Notation). 
For $m\in \Np$, define 
$\Z/m(0) := \Z/m$ with the trivial action of $G_F$, and 
$\Z/m(n) := \mu_m(\Fbar)^{\otimes n}$ 
for $n\ge 1$, 
where 
$\mu_m(\Fbar)$ is the Galois module of  
$m$-th roots of unity in $\Fbar$.  
We define (following \cite{Kat80}, Sect.~3.2, Def.\ 1)
\[
	H^0(F)  := \ilim_{m\in \Np} H^0_{\Gal}(F, \Z/m(-1)),
\] 	
where $\Z/m(-1) := \Hom(\mu_m(\Fbar),\QZ)$ 
on which $G_F$ acts by $f \mapsto f\circ \sigma^{-1}$ for $\sigma \in 
G_F, f \in \Z/m(-1)$
(\Cf \cite{Kat80}, Sect.\ 1.2). 
For $n\in \Z_{\ge 1}$, 
\[	
H^n(F) :=  \ilim_{m \in \Np} H^n_{\Gal}(F, \Z/m (n-1)) \oplus 
	\ilim_{r \in \Z_{\ge 1}} H^n_{\Gal}(F,\Z/p^r(n-1)).
\]
Using these, 
it is known 
that we have
\begin{equation}
\label{eq:GFv}
  H^1(F) \simeq (G_F)^{\vee} \simeq (G_F^{\ab})^{\vee} 
\end{equation}
(\Cf \cite{Kat80}, Sect.\ 3.2; see also \cite{Ras95}, Chap.\ 2). 
From this isomorphism, we identify $H^1(F)$ and $(G_F^{\ab})^{\vee}$ in the following.

\subsection*{$2$-dimensional local class field theory}
We recall the $2$-dimensional local class field theory 
following \cite{Kat79} and \cite{Kat80}. 
For detailed expositions on this section, 
we also recommend \cite{Ras95}, Chap.\ 2. 

\begin{dfn}
\label{def:2dim}
A {\bf $2$-dimensional local field}
is a complete discrete valuation field whose residue field is a local field.
\end{dfn}

Throughout this section, 
we fix such a field and use the following notation:  

\begin{itemize}
\item $K$:\, a $2$-dimensional local field of $\Char(K) = p$, 
\item $\vK:K^{\times} \to \Z$: the valuation of $K$, 
\item $\OK:=\set{f \in K | \vK(f)\ge 0}$:\, the valuation ring of $K$,
\item $\mK:=\set{f \in K | \vK(f) >0}$:\, the maximal ideal of $\OK$,  
\item $k:= \OK/\mK $:\, the residue field of $K$, and 
\item $U_K := \OKt$:\, the group of units in $\OK$.    
\end{itemize}

%
The class field theory of $K$ describes 
the abelian Galois group $G_K^{\ab} = \Gal(K^{\ab}/K)$ by  
a canonical homomorphism 
$\rho_K:K_2(K) \to  G_K^{\ab}$ 
called the \textbf{reciprocity map} (defined in \cite{Kat80}, Sect.\ 3.2). 

\begin{prop}[\cite{Kat80}, Sect.\ 3.2, Cor.\ 1 and 2; see also \cite{Ras95}, Sect.\ 2.1]
\label{prop:rhoK}
\begin{enumerate}
\item 
  	We have the following commutative diagram
  	\[
  		\xymatrix{
    	K_2(K) \ar[r]^-{\rho_K}\ar[d]_{\d_K}&  G_K^{\ab} \ar[d] \\
    	k^{\times} \ar[r]^-{\rho_k} & G_k^{\ab},
  		}
  	\]
  	where  $\rho_k$ is the reciprocity map of $k$, 
  	the right vertical map is the restriction,  
  	and  
	$\d_K$ is the \textbf{boundary map} 
defined by 
\begin{equation}
\label{eq:tame_symbol}
\xymatrix{
 \d_K(\set{f,g}) := (-1)^{v_K(f)v_K(g)}f^{v_K(g)}g^{-v_K(f)} \bmod \m_K, 
}
\end{equation} 
for $\set{f,g}\in K_2(K)$. 

\item 
For a finite extension $L/K$, the following diagram is commutative:
\[
  \xymatrix{
  K_2(K) \ar[r]^-{\rho_K} & G_K^{\ab}\\
  K_2(L) \ar[r]^-{\rho_L}\ar[u]^{N_{L/K}} & G_L^{\ab}\ar[u]_{\Res_{L/K}}, 
  }
\]
where $N_{L/K}$ is the norm map, 
and the right vertical map $\Res_{L/K}$ is the restriction. 

\item 
For a finite extension $L/K$, the following diagram is commutative:
\[
  \xymatrix{
  K_2(K) \ar[r]^{\rho_K}\ar[d]_{i_{L/K}} & G_K^{\ab}\ar[d]^{\Ver_{L/K}}\\
  K_2(L) \ar[r]^{\rho_L} & G_L^{\ab}, 
  }
\]
where $i_{L/K}$ is the map induced from the inclusion $K \inj L$, 
and the right vertical map $\Ver_{L/K}$ is the 
transfer map (\cite{NSW}, Sect.\ 1.5).
\end{enumerate}
\end{prop}

The multiplicative group $K^{\times}$ and the Milnor $K$-group 
$K_2(K)$ 
have good topologies 
(introduced in \cite{Kat79}, Sect.\ 7, see also \cite{Ras95}, Sect.\ 
2.3) and this makes $\rho_K$ continuous. 
We omit the detailed exposition on the definitions of these topologies. 
However, under the topologies, the following properties hold: 
\begin{enumerate}[label=$\mathrm{(\alph*)}$]
\item The unit group $U_K = \OKt$ is open in $\Kt$. 
\item The topology  on $K_2(K)$ is given by the strongest topology for the so called \textbf{symbol map} 
 $K^{\times} \times K^{\times} \to K_2(K); f\otimes g \mapsto \set{f,g}$
is continuous.  
\item For a finite extension $L/K$, the norm map (\cite{Kat80}, Sect.\ 1.7)  
$N_{L/K}:K_2(L) \to K_2(K)$
is continuous.
\end{enumerate}
Note also that 
any continuous homomorphism 
$K_2(K) \to \QZ$ is automatically of finite order 
with respect to this topology 
(\cite{Kat80}, Sect.~3.5, Rem.\ 4). 
%
Recall that an element $\chi\in H^1(K) = (G_K^{\ab})^{\vee}$ \eqref{eq:GFv}  
is said to be \textbf{unramified}   
if the corresponding cyclic extension of $K$  
is unramified. 
%
\begin{thm}[\cite{Kat80}, Sect.~3.1, 3.5; \cite{Sai85}, Chap.\ I, 
Thm.~3.1] 
	\label{thm:lcft}
	The reciprocity map $\rho_K$ satisfies the following: 
 	\begin{enumerate}
 	\item 
 	The map $\rho_K$ induces an isomorphism 
 	$\rho_K^{\vee}: H^1(K) \isomto K_2(K)^{\vee}$. 
%
 	\item
	An element $\chi \in H^1(K)$ is unramified  
	if and only if $\rho_K^{\vee}(\chi)$ 
	annihilates $U^0K_2(K) := \Ker(\partial_K)$.
	\end{enumerate} 
\end{thm}
We denote by $I_K$ the ineria subgroup of $G_K^{\ab}$ 
which is defined by the kernel of the restriction $G_K^{\ab} \to G_k^{\ab}$. 
For any $m\in \Z_{\ge 1}$, 
the reciprocity map $\rho_K$ induces 
$\rho_{K,m}:K_2(K)/m \to G_K^{\ab}/m$. 
Thm.\ \ref{thm:lcft} (i) implies that the dual of this homomorphism  
\begin{equation}
\label{eq:rhoKmv}
\xymatrix@C=5mm{
  \rho_{K,m}^{\vee}:(G_K^{\ab}/m)^{\vee} = H^1(K,\Z/m) \ar[r]^-{\simeq} &  
  (K_2(K)/m)^{\vee}
}
\end{equation}
is bijective. 
The following theorem says that $\rho_{K,m}$ is injective 
for each $m \in \Z_{\ge 1}$. 
%
%
\begin{thm}[\cite{Fes01}, Thm.\ 4.5, see also \cite{Fes00}, Thm.\ 2]
\label{thm:Fesenko}
$\Ker(\rho_K)$ 
is divisible.  
\end{thm}

\subsection*{Ramification theory}
For $m \in \Z_{\ge 1}$, 
let $U_K^m = 1+ \m_K^m$ be the higher unit groups of $K$.
Denote by $U^mK_2(K)$ the subgroup of $K_2(K)$ 
generated by the image of $U_K^m\times K^{\times}$ in $K_2(K)$ by the symbol map. 
We also have an increasing filtration $\set{\fil_mH^q(K)}_{m \in \Z_{\ge 0}}$ 
on $H^q(K)$ (\cite{Kat89}, Def.\ 2.1) 
with $H^q(K) = \cup_{m \in \Z_{\ge 0}}\fil_mH^q(K)$. 
In particular, we have $\fil_0H^1(K) \simeq H^1(k) \oplus H^0(k)$ 
and this subgroup corresponding to tamely ramified abelian extensions of $K$ 
(\cite{Kat80}, Thm.\ 3; \cite{Kat89}, Prop.\ 6.1). 
This filtration on $H^1(K)$ induces the ramification filtration 
$\set{I^m_K}_{m \in \Z_{\ge 0}}$ on 
$G_K^{\ab}$, which is defined by 
$I_K^0 := I_K$ and 
\[
I_K^m :=   \set{\sigma \in G_K^{\ab} | \chi(\sigma) = 0 \mbox{\ for all $\chi \in \fil_{m-1}H^1(K)$} }
\]
for $m\ge 1$. 
The description of $\fil_0H^1(K)$ implies that 
$I_K^m \subset I_K = I_K^0$ for $m\ge 1$ and 
$I_K^1$ is the wild inertia subgroup of $G_K^{\ab}$, that is, 
the maximal pro-$p$ subgroup of the inertia subgroup $I_K$. 
  
\begin{prop}[\cite{Kat89}, Prop. 6.5, see also Rem.\ 6.6]
\label{prop:Kato}	
For $\chi \in H^1(K)$, 
$\chi$ is in $\fil_mH^1(K)$ if and only if 
$\rho_K^{\vee}(\chi) \in K_2(K)^{\vee}$ 
annihilates $U^{m+1}K_2(K)$.
\end{prop}

From Prop.\ \ref{prop:Kato}, $\rho_K$ induces 
$U^{m}K_2(K) \to  I_K^{m}$ 
%
for $m\in \Z_{\ge 0}$.
In our case of $\Char(K) = p$, 
it is known  
$I_K^{m+1} = G_{K,\mathrm{log}}^{\ab, m+}$ 
for any $m\in \Z_{\ge 0}$, 
where the right is the induced group from 
Abbes-Saito's 
logarithmic version of ramification subgroups on  
the absolute Galois group $G_K = \Gal(\Kbar/K)$ 
(\cite{AS02}, see also \cite{AS09}, Cor.\ 9.12). 

\section{Curves over local fields}
\label{sec:class groups}
Let $k$ be a local field of $\Char(k) = p$. 
%

\begin{dfn}
\label{def:open_curve}
We call the pair $X \subset \Xbar$ of 
\begin{itemize} 
\item $\Xbar$\,:  a smooth, proper and geometrically connected curve over $k$,  and 
\item $X$\,: a nonempty open subscheme of $\Xbar$
\end{itemize}
an \textbf{open curve} over $k$.
\end{dfn}

\noindent
Since 
the smooth compactification $\Xbar$ of a smooth curve $X$ is unique 
if it exists by the valuative criterion of properness, 
we often omit $\Xbar$ 
and write $X$ solely as an open curve in the above sense.

For an open curve $X$ over $k$, 
we also define 
\begin{itemize} 
\item $X_{\infty} := \Xbar \ssm X$, 
\item $X_0$\,: the set of closed points in $X$, and 
\item $\kX$\,: the function field of $X$.
\end{itemize}

\sn
For a closed point $x\in \Xbar_0$, 
we denote by 
\begin{itemize} 
\item $\kx:$ the residue field at $x$ 
which is a finite extension of $k$, and 
\item $\kXx:$ the completion of $\kX$ at $x$ 
which is a 2-dimensional local field (Def.\ \ref{def:2dim}) 
with residue field $\kx$.
\end{itemize}

\subsection*{Id\`ele Class Groups}
We fix an open curve $X$ over $k$
and introduce the id\`ele class group and the reciprocity map 
for $X$.   

\begin{dfn}
\label{def:CX}
The {\bf id\`ele class group} $\CX$ 
is defined to be the cokernel of 
\[
\xymatrix@C=5mm{
  \d:K_2(\kX) \ar[r] & \displaystyle{\bigoplus_{x\in X_0} \kxt \oplus \bigoplus_{x\in 
  \Xinf}K_2(\kXx) }
}
\]
which is given by 
the direct sum of the following homomorphisms:  
\begin{itemize}[labelindent=\parindent,leftmargin=*]
\item  
the boundary map \eqref{eq:tame_symbol}  
$\d_x := \d_{\kXx} :K_2(\kXx) \to \kxt$ 
for $x \in X_0$, and 

\item $K_2(\kX) \to  K_2(\kXx)$ 
induced by  
the inclusion $i_x:\kX \inj \kXx$ 
for $x \in \Xinf$, and 
\end{itemize} 
\end{dfn}

The restricted product 
$\restprod_{x\in \Xbar_0} K_2(\kXx)$
with respect to 
the closed subgroup $U^0K_2(\kXx) = \ker(\d_x)$
has a structure of a topological group induced from 
the topology on $K_2(\kXx)$ (\Cf Sect.\ \ref{sec:local}) 
as in the classical class field theory (\Cf \cite{Sai85}, Chap.\ I, Sect.\ 3). 
The id\`ele class group  $\CX$ is 
a quotient of $\restprod_{x\in \Xbar_0} K_2(\kXx)$ 
and is endowed with the quotient topology.

The abelian fundamental group $\piabX$ 
has a description as a Galois group: 
we have $\piabX \simeq \Gal(\kX^{\ur}/\kX)$, 
where $\kX^{\ur}$ is 
generated by all finite separable extensions $E$ of $\kX$ contained in $\kX^{\ab}$ 
satisfying that  
the normalization $\widetilde{X}^E \to X$ of $X$ in $E$ is unramified (\Cf \cite{SGA1} Exp.\ V, 8.2). 
In particular, we have $\kX^{\ur} \subset \kX^{\ab}$ so that 
the restriction gives 
$G_{\kX}^{\ab}=\Gal(\kX^{\ab}/\kX) \surj \piabX$.
The 2-dimensional local class field theory 
$\rho_{\kXx}: K_2(\kXx) \to G_{\kXx}^{\ab}$ 
and the restriction $G_{\kXx}^{\ab} \to G_{\kX}^{\ab}$  
induce a continuous 
homomorphism
\[
\xymatrix@C=5mm{
\displaystyle{
	\restprod_{x\in \Xbar_0} K_2(\kXx) }\ar[r]&  G_{\kX}^{\ab} \ar@{->>}[r]& \piabX.
} 
\] 
By the reciprocity law of $\kX = k(\Xbar)$ (\cite{Sai85}, 
Chap.\ II, Prop.\ 1.2) 
and the $2$-dimensional local class field theory (Thm.\ 
\ref{thm:lcft}), 
this factors through $\CX$.  

\begin{dfn}
\label{def:rhoX}
The induced continuous homomorphism 
$\rhoX : \CX \to \piabX$
%
is called the \textbf{reciprocity map} of $X$. 
\end{dfn}

We denote by 
\begin{equation}
\label{def:cok}
  \Cokertop(\rhoX) := \piabX/\ol{\Im(\rhoX)},
\end{equation}
where $\ol{\Im(\rhoX)}$ is the topological closure of $\Im(\rhoX)$. 

The norm map $N_{k(x)/k}:k(x)^{\times} \to k^{\times}$ for $x\in 
X_0$ 
and the composition $N_{k(x)/k}\circ \dx:K_2(\kXx)\to k^{\times}$ for $x \in \Xinf$ 
induce a homomorphism 
$N_X:\CX \to k^{\times}$. 
They make the following diagram commutative: 
\begin{equation}
\label{eq:VX}
\vcenter{
\xymatrix{
  0 \ar[r] & \CX^0 \ar[r]\ar[d]^{\rhoX^0}& \CX \ar[r]^-{N_X}\ar[d]^{\rhoX} & 
  k^{\times} \ar[d]^{\rho_k}\\
  0\ar[r]  & \piabX^0 \ar[r] & \piabX \ar[r]^-{\varphi} & \piab(\Spec(k)) = G_k^{\ab}\ar[r] & 0,
}}
\end{equation}
where $\varphi$ is the induced homomorphism from the structure morphism $X\to \Spec(k)$ 
(\cite{Fu11}, Sect.\ 3.3) 
and the groups $\CX^0$ and $\piabX^0$ are defined by the exactness of the 
horizontal rows. 

\subsection*{Restricted ramification}
For the open curve $X$, 
let $D = \sum_{x\in \Xinf} m_x [x]$ 
be an effective Weil divisor on $\Xbar$ 
with support $|D| \subset  \Xinf = \Xbar\ssm X$. 

\begin{dfn} 
\label{def:piabXD}
By putting $m_x = 0$ if $x \not \in |D|$, 
we define the abelian fundamental group $\piabXD$ 
with bounded ramification by  
\[
\xymatrix@C=5mm{
  \piabXD = \Coker
  \Bigg(\,\displaystyle{\bigoplus_{x\in \Xinf} I_{\kXx}^{m_x} } \ar@{^{(}->}[r] 
   & \displaystyle{\bigoplus_{x\in \Xinf} G_{\kXx}^{\ab}} \ar[r] & \piabX\bigg),
}
\]
where $I_{\kXx}^{m_x}$ is the ramification subgroup of 
$G_{\kXx}^{\ab} = \Gal(\kXx^{\ab}/\kXx)$  
(Sect.\ \ref{sec:local}). 
\end{dfn}
%
By Prop.\ \ref{prop:Kato}, 
the composite $\CX\onto{\rhoX} \piabX\surj \piabXD$ 
factors through 
\[
  \xymatrix@C=5mm{
  \CXD := \Coker\Bigg(\, \displaystyle{\bigoplus_{x\in \Xinf} U^{m_x}K_2(\kXx) }\ar[r] & \CX\Bigg)
 } 
\]
and the induced homomorphism is denoted by 
$\rhoXD:\CXD \to \piabXD$. 
Furthermore, the norm maps $N_{k(x)/k}:k(x)^{\times} \to k^{\times}$ 
define $N_{X,D}:\CXD \to k^{\times}$ and the following  diagram 
is commutative as in \eqref{eq:VX}: 
\begin{equation}
\label{eq:VXD}
	\vcenter{
	\xymatrix{
	  0 \ar[r] & \CXDg \ar[r]\ar[d]^{\rho_{X,D}^0}& \CXD 
	  \ar[r]^-{N_{X,D}}\ar[d]^{\rhoXD} & k^{\times} \ar[d]^{\rho_k}\\
	  0\ar[r]  & \piabXD^0 \ar[r] & \piabXD \ar[r] & G_k^{\ab} \ar[r] & 0.
	}}
\end{equation}
Here, 
the groups $\CXDg$ and $\piabXD^0$ are defined by the exactness of the 
horizontal rows. 
%

Consider $\Xinf = \sum_{x\in \Xinf}1[x]$ as a Weil divisor. 
Recalling that 
$I_{\kXx}^{1}$ is the wild inertia subgroup, the groups 
\begin{equation}
\label{eq:tfg}
\pitabX := \piab(X, \Xinf),\quad \mbox{and}\quad \pitabXg := 
\piab(X,\Xinf)^0
\end{equation}
classify \textbf{tame coverings} of $X$, 
that is, finite \'etale coverings over $X$ and ramify at most tamely along the boundary $\Xinf$.  
We also employ the following notation:  
\begin{equation}
\label{eq:rhot}
\vcenter{
\xymatrix@C=5mm@R=0mm{
  \rhoXt := \rho_{X,\Xinf} :\CtX := C(X,\Xinf) \ar[r] & \pitabX, &\mbox{and}\\
  \rhoXtg := \rho_{X,\Xinf}^0:\CtXg := C(X,\Xinf)^0 \ar[r] & \pitabXg.
}}
\end{equation}

\subsection*{Functorial properties}
We define the pullback and the norm homomorphism
on the id\`ele class groups with respect to \'etale coverings of open curves in the following sense. 

\begin{dfn}
\label{def:cover}
An  \textbf{\'etale covering} 
$f:Y\to X$ 
of open curves 
is defined to be the commutative diagram 
\begin{equation}
\label{eq:morphism}
\vcenter{
\xymatrix@R=5mm{
   Y\ar@{^{(}->}[r]\ar[d]_f &\Ybar \ar[d]^{\fbar} & \ar@{_{(}->}[l]\Yinf\ar[d] \\
   X\ar@{^{(}->}[r] & \Xbar   & \Xinf, \ar@{_{(}->}[l]
}}
\end{equation}
where the horizontal maps are the inclusions, $\fbar$ is a morphism of schemes over $\Spec(k)$ 
and, $f$ is a finite \'etale morphism of schemes over $\Spec(k)$. 
The right commutative square in \eqref{eq:morphism} means $\fbar(\Yinf) \subset \Xinf$.   
%
\end{dfn}
In the following, we fix 
$f:Y\to X$ an \'etale covering of open curves over $k$. 
%


\begin{dfn}
We define a canonical homomorphism 
$i_{Y/X}:= f^{\ast}:\CX \to \CY$
%
as follows: 
\begin{itemize}
\item For $x \in X_0$ and $y \in Y_0$ with $f(y) = x$,  
the inclusion $\kx \inj \ky$ gives 
$i_{\ky/\kx}: \kxt \inj \kyt$.
%

\item 
For $x \in \Xinf$, 
and $y \in \Yinf$ with $\fbar(y) = x$, 
the inclusion map $\kXx \inj \kYy$ gives 
$i_{\kYy/\kXx}: K_2(\kXx) \to K_2(\kYy)$.
%
\end{itemize}
These maps give a canonical homomorphism
\[
\xymatrix@C=5mm{
	\displaystyle{\bigoplus_{x\in X_0}\kxt \oplus \bigoplus_{x 
\in \Xinf}K_2(\kXx) }\ar[r] & \displaystyle{\bigoplus_{y \in Y_0}\kyt \oplus \bigoplus_{y\in 
\Yinf}K_2(\kYy)}. 
}
\]
Since the homomorphism $K_2(\kX) \to K_2(\kY)$ induced from $\kX \inj \kY$ 
is compatible with above homomorphisms, 
we obtain $i_{Y/X}$. 
\end{dfn}

\begin{dfn}
\label{def:norm}
We define the {\bf norm map} 
$N_{Y/X}:= f_{\ast}:\CY \to \CX$ 
as follows: 

\begin{itemize}
\item 
For $y \in Y_0$ with $x = f(y)$, 
we have the norm homomorphism 
$N_{\ky/\kx}: \kyt \to \kxt$.  
%

\item 
For $y \in \Yinf$ with $x = \fbar(y)$, 
we have the norm map 
$N_{\kYy/\kXx}: K_2(\kYy) \to K_2(\kXx)$. 
\end{itemize}
\noindent
These maps give a canonical homomorphism
\[
\xymatrix@C=5mm{
	\displaystyle{\bigoplus_{y \in Y_0}\kyt \oplus \bigoplus_{y\in 
\Yinf}K_2(\kYy) }\ar[r] & \displaystyle{\bigoplus_{x\in X_0}\kxt \oplus \bigoplus_{x 
\in \Xinf}K_2(\kXx)}. 
}
\]
Since the norm $N_{\kY/\kX}:K_2(\kY) \to K_2(\kX)$ is compatible with above norms, 
we obtain 
$N_{Y/X}$. 
\end{dfn}

\begin{lem}
\label{lem:pf}
We  have $N_{Y/X} \circ i_{Y/X} = [k(Y):k(X)] \cdot  \Id_{\CX}$,
where $\Id_{\CX}$ is the identity map of $\CX$.
\end{lem}
\begin{proof}
	The projection formula of the Milnor $K$-groups (e.g., \cite{Mil71}, 
	Sect.\ 14) gives  
	\[
	  N_{\kYy/\kXx}\circ i_{\kYy/\kXx} = [\kYy:\kXx]\cdot \Id_{K_2(\kXx)}. 
	\]  
	The assertion follows from the equality
	\[
	  [\kY:\kX] = \sum_{y\in {\fbar}^{-1}(x)}[\kYy:\kXx]
	\]
	for a closed point $x\in \Xbar_0$ 
	(\cite{Ser68}, Chap.\ I, Sect.\ 4, Prop.\ 10).  
\end{proof}

From the construction of $\rhoX$ and the properties of $\rho_{K_x}$ for each $x\in \Xbar_0$ given in Prop.\ \ref{prop:rhoK},  
we obtain the following commutative diagrams: 
\begin{equation}
\label{eq:norm}
\vcenter{
\xymatrix@R=1mm{
	\CX \ar[r]^{\rhoX} & \piabX         & &	\CX \ar[r]^{\rhoX}\ar[dd]_{i_{Y/X}} & \piabX\ar[dd]^{\psi} \\
	& &                           \mbox{and} &  \\
	\CY \ar[r]^{\rho_Y}\ar[uu]^{N_{Y/X}} & \piabY\ar[uu]_{\varphi} & & \CY \ar[r]^{\rho_Y} & \piabY
}
}
\end{equation}
%
where $\varphi$ is the induced homomorphism 
of the fundamental groups from $f$ 
and $\psi$ is given by the transfer map. 

\section{Proof of Thm.\ \ref{thm:main2intro}}
\label{sec:main}
In this section, we prove Thm.\ \ref{thm:main2intro} 
using the following notation: 
\begin{itemize}
\item $k$\,: a local field of $\Char(k) = p$, and  
\item $X \subset \Xbar$\,: an open curve over $k$ in the sense of Def.\ \ref{def:open_curve}. 
\end{itemize}

\subsection*{Unramified class field theory}
We recall the class field theory for 
the projective smooth curve $\Xbar$ following \cite{Sai85} and \cite{Yos03}. 
Note that the id\`ele class groups $\CXbar$ and $\CXbarg$ are denoted by 
$SK_1(\Xbar)$ and $V(\Xbar)$ respectively in \opcit
%

\begin{thm}[\cite{Sai85}, Chap.\ II, Thm.\ 2.6, 5.1, Prop.\ 3.5, and Thm.\ 4.1;  \cite{Yos03}, Thm.\ 5.1]
\label{thm:ucft}
For the reciprocity map $\rhoXbar:\CXbar \to \piabXbar$, 
we have: 
\begin{enumerate}
%
\item $\Cokertop(\rhoXbar) \simeq \Zhat^{\rXbar}$ 
for some $\rXbar \ge 0$, 

\item $\Ker(\rhoXbar)$ and $\Ker(\rhoXbarg)$ are   
the maximal divisible subgroups of $\CXbar$ and $\CXbarg$ respectively, 

\item $\# \Im(\rhoXbar^0) < \infty$, and $\Coker(\rhoXbar^0) \simeq \Zhat^{\rXbar}$. 
\end{enumerate}
\end{thm} 
\noindent
Here, the invariant $\rXbar$ is determined by 
the special fiber of the N\'eron model of the Jacobian variety of 
$\Xbar$ 
which is called the \textbf{rank} of $\Xbar$ 
(\cite{Sai85}, Chap.\ II, Def.\ 2.5, see also \opcit, 
Chap.\ II, Thm.\ 6.2). 

Thm.\ \ref{thm:ucft} gives  
the 
structure of the fundamental group $\piabXbarg$ as in the following short exact sequence:  
\begin{equation}
\label{eq:rho0}
\xymatrix@C=5mm{
	0 \ar[r] & \piabXbarg_{\tor}= \Im(\rhoXbarg) \ar[r] & \piabXbarg \ar[r] & \Zhat^{\rXbar} \ar[r] & 0,
}
\end{equation}
where 
$\piabXbarg_{\tor}$ is the torsion part of $\piabXbarg$ 
which is finite. 
%
%
%

\subsection*{Tame fundamental groups}
The goal of this paragraph is 
to determine 
the structure of the abelian tame fundamental group 
 $\pitabX = \piab(X, \Xinf)$ \eqref{eq:tfg} as in \eqref{eq:rho0}. 
%
%

\begin{thm}
\label{thm:coker}
$\Cokertop(\rhoX) \simeq \Zhat^{\rXbar}$,  
where $\rXbar$ is the rank of $\Xbar$. 
\end{thm}
\begin{proof} 
For any $x \in \Xinf = \Xbar \ssm X$, put 
$Y := \Spec(\O_{\Xbar,x}^{\wedge})$, where 
$\O_{\Xbar,x}^{\wedge}$ is the completion of the local ring $\O_{\Xbar,x}$. 
The localization sequence of the \'etale cohomology groups on 
$i:x \inj Y$ 
(\cite{Fu11}, Prop.\ 5.6.12)
gives an exact sequence 
\[
\xymatrix@C=5mm{
	0 \ar[r] & H^1(Y,\QZ) \ar[r] & H^1(\Spec(\kXx),\QZ) \ar[r] &  H^2_{x}(Y,\QZ) \ar[r] & H^2(Y,\QZ).
}\]
In terms of the Galois cohomology groups (\cite{Fu11}, Prop.\ 5.7.8), 
we have 
\[
\xymatrix@C=5mm{
	H^n(Y,\QZ) \ar[r]_-{i^{\ast}}^-{\simeq} & H^n(x,\QZ) \simeq 
	\HGal^n(k(x),\QZ)
}
\]
(the first isomorphism follows from \cite{SGA4}, Exp.\ XII, Rem.\ 6.13) 
and $H^n(\Spec(\kXx),\QZ) \simeq \HGal^n(\kXx,\QZ)$.
By the Tate duality theorem for local fields 
(\cite{NSW}, Thm.\ 7.2.6) (for prime to the $p$-part) 
and the dimension reason (\cite{NSW}, Prop.\ 6.5.10) (for the $p$-part),    
we have
\begin{equation}
\label{eq:H2(k)}
\HGal^2(\kx,\QZ) = 0.
\end{equation}
The excision theorem induces $H^2_x(Y,\QZ) \simeq H^2_x(\Xbar, \QZ)$ 
(\Cf \cite{Fu11}, Prop.\ 5.6.12). 
We also have 
$H^1_{\Gal}(\kXx,\QZ) \simeq H^1(\kXx)$ \eqref{eq:GFv}. 
Thus, we obtain the commutative diagram below: 
\begin{equation}
\label{eq:locx}
\vcenter{
  \xymatrix{
 0 \ar[r] & H^1(\kx) \ar[d]_{\simeq}^{\rho_{\kx}^{\vee}}\ar[r] & H^1(\kXx) 
 \ar[d]^{\rho_{\kXx}^{\vee}}_{\simeq} \ar[r]&  H^2_x(\Xbar,\QZ) 
 \ar[r]\ar@{-->}[d]^{\phi_x} & 0 \\
 0 \ar[r] & (\kxt)^{\vee} \ar[r]^-{\dx^{\vee}} & K_2(\kXx)^{\vee} 
 \ar[r]& U^0K_2(\kXx)^{\vee}  & ,\\
 }}
\end{equation}
where  
$\rho_{\kx}$ and $\rho_{\kXx}$ are the reciprocity maps of $\kx$ and 
$\kXx$ respectively 
(Thm.\ \ref{thm:lcft}). 
Here, the bottom sequence is exact.  
%
%
Using a canonical isomorphism 
\begin{equation}
  \label{eq:H1pi}
  H^1(X,\QZ) \simeq \piab(X)^{\vee} 
\end{equation}
(\cite{SGA4.5}, Exp.\ 1, Sect.\ 2.2.1, or \cite{Fu11}, Prop.\ 5.7.20),  
%
we consider the following commutative diagram:  
\begin{equation}
\label{eq:loc}
\vcenter{
\entrymodifiers={!! <0pt, .8ex>+}
 \xymatrix@R=5mm@C=5mm{ 
 0\ar[r] & H^1(\Xbar, \QZ) \ar[d]^{\rhoXbarv}\ar[r] & H^1(X, \QZ) 
 \ar[d]^{\rhoXv}\ar[r] 
   & \displaystyle{\bigoplus_{x\in \Xinf}H^2_{x}(\Xbar, \QZ)} 
   \ar[d]^-{\oplus\, \phi_x}  \\
  0 \ar[r] & \CXbarv  \ar[r] &  \CXv \ar[r]^-{i} 
   & \displaystyle{\bigoplus_{x\in \Xinf}U^0K_2(\kXx)^{\vee}}.
 }
}
\end{equation}
Here, the upper horizontal sequence is the localization sequence 
associated to $\Xinf \inj \Xbar$. 
The diagram \eqref{eq:loc} 
gives 
$\Ker(\rhoXbarv) \simeq \Ker(\rhoXv)$.  
By Thm.\ \ref{thm:ucft} (i), we obtain 
\[
	\Zhat^{\rXbar} \simeq \Cokertop(\rhoXbar) \simeq 
	\Ker(\rhoXbarv)^{\vee} 
	\simeq \Ker(\rhoXv)^{\vee} \simeq \Cokertop(\rhoX). 
\]
%
%
The assertion follows from this. 
\end{proof}

For any effective Weil divisor $D$ on $\Xbar$ 
whose support $|D| \subset \Xinf$, 
we have canonical surjective homomorphisms 
\[
\xymatrix@C=5mm{
  \piabX \ar@{->>}[r] & \piabXD \ar@{->>}[r] & \piabXbar 
}\]
from the very definition of $\piabXD$ (Def.\ \ref{def:piabXD}). 
The above Thm.\ \ref{thm:coker} and Thm.\ \ref{thm:ucft} (i) imply also 
\begin{equation}
\label{eq:cok(rhoXt)}	
\Cokertop(\rhoXD) := \piabXD/\ol{\Im(\rhoXD)} 
\simeq  
\Zhat^{\rXbar}.
\end{equation}
\begin{lem}
\label{lem:rhoXtg}
For the map 
$\rhoXtg:\CtXg\to \pitabXg$ $\mathrm{\eqref{eq:rhot}}$, we have 
$\#\Im(\rhoXtg) < \infty$.
\end{lem} 
\begin{proof}
	For each $x\in \Xinf$, 
	let $I_{\kXx} = I_{\kXx}^0$  be the inertia subgroup of $G_{\kXx}^{\ab}$, that is, 
	the kernel of the restriction 
	$G_{\kXx}^{\ab} \to G_{\kx}^{\ab}$. 
	Thm.\ \ref{thm:lcft} and Prop.\ \ref{prop:Kato} imply that 
	$\rho_{\kXx}$ induces 
	$U^0K_2(\kXx)/U^1K_2(\kXx) \to I_{\kXx}^0/I_{\kXx}^1$.
%
	This gives the following commutative diagram with exact rows: 
	\[
		\entrymodifiers={!! <0pt, .8ex>+}
		\xymatrix@C=5mm@R=5mm{
			\displaystyle{\bigoplus_{x\in \Xinf} U^0K_2(\kXx)/U^1K_2(\kXx)} \ar[r] 
			\ar[d] & \CtXg 
			\ar[r]\ar[d]^{\rhoXtg}& 
			\CXbarg \ar[d]^{\rhoXbarg} \ar[r]& 0\ \\
			 \displaystyle{\bigoplus_{x\in \Xinf} I_{\kXx}^0/I_{\kXx}^1} \ar[r] & 
			 \pitabXg 
			\ar[r]^-{\varphi}
			\ar[r] & \piabXbarg\ar[r] & 0,
		}
	\]
	where $\varphi$ is the induced homomorphism 
	from the open immersion $X\inj \Xbar$. 
	For each $x\in \Xinf$, 
	we have 
	\begin{itemize}
	\item $\dx:K_2(\kXx)/U^0K_2(\kXx)  \isomto \kxt$ (by $U^0K_2(\kXx) = \Ker(\dx)$), and 
	\item $K_2(\kXx)/U^1K_2(\kXx) \simeq K_2(\kx)\oplus \kxt$ 	(\Cf \cite{FV}, Chap.\ IX, Prop.\ 2.2). 
	\end{itemize}  
	These isomorphisms give 
	$U^0K_2(\kXx)/U^1K_2(\kXx) \simeq  K_2(\kx)$. 
	By Merkrjev's theorem (\cite{FV}, Chap.\ IX, Thm.\ 4.3), 
	$K_2(\kx)$ is the sum of a finite group and a divisible subgroup. 
	By Thm.\ \ref{thm:lcft}, 
	$\rho_{\kXx}^{\vee}$ induces an injective homomorphism  
	$(I_{\kXx}^0/I_{\kXx}^1)^{\vee} \inj (U^0K_2(\kXx)/U^1K_2(\kXx))^{\vee}$.  
	Therefore, the quotient $I_{\kXx}^0/I_{\kXx}^1$ is finite 
	and  so is $\Ker(\varphi)$. 
	The assertion $\# \Im(\rhoXtg) < \infty$ 
	follows from $\#\Im(\rhoXbar^0) < \infty$ (Thm.\ \ref{thm:ucft}, 
	(iii)).   
\end{proof}
From Lem.\ \ref{lem:rhoXtg} and \eqref{eq:cok(rhoXt)}, 
we have a short exact sequence 
\begin{equation}
\label{eq:pitabXg}
\xymatrix@C=5mm{
  0\ar[r] & \pitabX^0_{\tor} = \Im(\rhoXtg) \ar[r] & \pitabXg \ar[r] &  
  \Zhat^{\rXbar} \ar[r] & 0 .
  }
\end{equation}

\subsection*{Open curves}

The rest of this section is devoted to show 
Thm.\ \ref{thm:main2intro} (ii) (=Thm.\ \ref{thm:main2} below).    
Recall $\Np = \set{m \in \Z_{\ge 1} | \mbox{$m$ is prime to $p$}}$, 
and, 
the reciprocity map $\rhoX$  induces 
$\rhoXm: \CX/m \to \piabX/m$ for each $m\in \Z_{\ge 1}$.

\begin{lem}
\label{lem:rhoXm}
For any $m\in \Np$,  
$\rhoXm:\CX/m\to \piabX/m$ 
is injective.
\end{lem}
\begin{proof}
	For any $m \in \Np$, 
	we have 
	$H^3_c(X,\Z/m(2)) = H^3(\Xbar,j_!\Z/m(2))$ 
	(\cite{Fu11}, Sect.\ 7.4), 
	where $\Z/m(n) = \mu_m^{\otimes n}$  and 
	$j:X \inj \Xbar$ is the open immersion.	
	We define a commutative diagram: 
	\[
   	\entrymodifiers={!! <0pt, .8ex>+}
   	\xymatrix@C=-2mm@R=4mm{
    	K_2(\kX)/m \ar[r]\ar[d]^-{h} & \displaystyle{\bigoplus_{x\in 
    	X_0}}\kx^{\times}/m \oplus 
    \displaystyle{\bigoplus_{x\in \Xinf}} K_2(\kXx)/m  \ar[r] \ar[d] & \CX/m 
    \ar[r]\ar[d] & 0 \\
    \HGal^2(\kX,\Z/m(2)) \ar[r] & \displaystyle{\bigoplus_{x\in 
    \Xbar_0}}H^3_x(\Xbar, j_!\Z/m(2))\ar[r]&  H^3(\Xbar,j_!\Z/m(2)) &.  
    }
	\]
	\vspace{-10mm}
	
	\noindent
	Here, the horizontal sequences are exact, 
	and the left vertical map $h$  
	is bijective by the Merkurjev-Suslin theorem \cite{MS82}.
	The middle vertical map is also bijective 
	from the following facts:  
	\begin{itemize} 
	\item For $x \in X_0$, the Kummer theory gives 
	$K_1(\kx)/m \isomto \HGal^1(\kx,\Z/m(1)) \simeq H^3_{x}(\Xbar,j_!\Z/m(2))$,
	where the latter isomorphism follows from 
	the excision theorem (\cite{Fu11}, Prop.\ 5.6.12): 
	$H^3_x(\Xbar,j_{!}\Z/m(2)) \simeq H^3_x(X,\Z/m(2))$,  
	the purity theorem (\cite{Fu11}, Cor.\ 8.5.6) 
	for the closed immersion $i:x\inj X$:  
	\[
	R^ti^{!}\Z/m(2) = \begin{cases} 
	0 ,& t\neq 2, \\
	i^{\ast}\Z/m(1), & t= 2, 
	\end{cases}
	\]
	and the Leray spectral sequence (\cite{Fu11}, Prop.\ 5.6.11): 
	$E_2^{s,t} = H^s(x, R^ti^{!}\Z/m(2)) \Rightarrow H^{s+t}_x(X,\Z/m(2))$.
	
	\item For $x \in \Xinf$, the Merkurjev-Suslin theorem again  gives  
	\[
	\xymatrix@C=5mm{
		K_2(\kXx)/m \ar[r]^-{\simeq} &  \HGal^2(\kXx,\Z/m(2)) \simeq 
		H^3_{x}(\Xbar,j_!\Z/m(2)). 
	}
	\]
	\end{itemize}
	Here, the latter isomorphism is given by the excision theorem (\cite{MilEC}, Chap.\ III, Cor.\ 1.28):
	\[
	  H^3_x(\Xbar, j_{!}\Z/m(2)) \simeq H^3_x(\Spec(\O_{\Xbar,x}^h), j_{!}\Z/m(2)), 
	\]
	and \cite{Mil06}, Chap.\ II, Prop. 1.1: 
	\[
	  H^3_x(\Spec(\O_{\Xbar,x}^h),j_{!}\Z/m(2)) \simeq \HGal^2(k(X)_x^h,\Z/m(2)) \simeq \HGal^2(\kXx,\Z/m(2)),
	\]
	where $\O_{\Xbar,x}^h$ is the henselization of the local ring $\O_{\Xbar,x}$, 
	and $k(X)_x^h$ is its fraction field. 
	Thus, the induced homomorphism $\CX/m\to H^3_c(X, \Z/m(2))$ is 
	injective from the above diagram. 
	By the duality theorem (\cite{Sai89}), 
	we have 
	$\piabX/m \simeq H^3_c(X,\Z/m(2))$ 
	so that $\rhoXm:\CX/m\to \piabX/m$ is 
	injective.  
\end{proof}

Before proving Thm.\ \ref{thm:main2intro} (ii) (=Thm.\ \ref{thm:main2} below), 
we prepare some notation (following \cite{For15}, Sect.\ 3) and quote a lemma from \cite{JS03}. 
For a set of primes $\L$ with $p \not \in \L$, define 
\begin{itemize}
\item $\NL:= \set{m \in \Np| \mbox{the prime divisors in $\L$}}$ as a sub monoid of $\Np$.  
\end{itemize}   
For an abelian  group $G$,
the natural surjective homomorphisms 
$G \to G/m$ for $m\in \NL$ induces a homomorphism 
\begin{equation}
\label{eq:whp}
\xymatrix@C=5mm{
\phi_{G,\L} : G\ar[r] &  G_{\L} := \displaystyle{\plim_{m \in \NL}G/m}.
}
\end{equation}

\begin{lem}[\cite{JS03}, Lem.\ 7.7]
\label{lem:JS}
Let $A$ be an abelian group, 
$\set{B_m}_{m \in \NL}$ a projective system of abelian groups, 
and a morphism  
$\set{\vp_m:A/m \to B_m}_{m \in \NL}$ of the projective systems. 
Put $B_{\L} :=  \plim_{m\in \NL}B_m$. 
If we assume that 

\begin{enumerate}[label=$\mathrm{(\alph*)}$]
\item $\vp_m$ is injective for all $m\in \NL$, and 
\item there exists $N \in \NL$ 
such that $N\cdot (B_{\L})_{\mathrm{tor}} = 0$,    
\end{enumerate}

\noindent
then  
$\Ker(\phi_{A,\L}: A \to A_{\L})$ 
is $l$-divisible for any prime $l \in \L$. 
\end{lem}

\begin{thm}
\label{thm:main2}
Let $k$ be a local field of $\Char(k) = p$, and 
$X\subset \Xbar$ an open curve over $k$. 
Then 
$\Ker(\rhoX)$ is 
the maximal $l$-divisible subgroup of $\CX$ for all prime number $l\neq p$. 
\end{thm}
\begin{proof} 
Since any profinite group does not contain non-trivial divisible elements, 
it is enough to show that, for any set of primes $\L$ with $p \not\in \L$, 
$\Ker(\rhoX)$ is $l$-divisible for all $l\in \L$. 
	From Lem.\ \ref{lem:rhoXm}, we have an injective homomorphism 
	$\rho_{X,\L} := \plim_{m\in \NL}\rhoXm: \CX_{\L} \inj \piabX_{\L}$ which commutes with $\rhoX$ as 
	in the following commutative diagram:  
	\[
	    \xymatrix{
		    \CX \ar[r]^{\rhoX}\ar[d]_{\psi} & \piabX \ar[d]^{\phi}\\
		    \CX_{\L} \ar@{^{(}->}[r]^{\rho_{X,\L}} & 
		    \piabX_{\L},
	    }
	\]
	where the vertical maps are the natural one $\psi = \phi_{\CX,\L}$ and $\phi = \phi_{\piabX,\L}$ defined in \eqref{eq:whp}. 
	This diagram gives an exact sequence 
	\begin{equation}
	\label{eq:Kerdiv}
	\xymatrix@C=5mm{
		0\ar[r] &  \Ker(\rhoX) \ar[r] & \Ker(\psi) \ar[r] & \Ker(\phi).  
	}\end{equation}

\begin{claim}
For any prime number $l\in \L$, we have 
\begin{enumerate} 
\item $\Ker(\psi)$ is $l$-divisible, and
\item $\Ker(\phi)$ is $l$-torsion free, that is, 
if we have $lx = 0$ with $x \in \Ker(\phi)$ then $x=0$.  
\end{enumerate}
\end{claim}
\begin{proof}
(i)  
Put $A := \CX, B_m := \piabX/m$ and $\vp_m := \rho_{X,m} : A/m \to B_m$. 
Using Lem.\ \ref{lem:JS}, 
we show that $\Ker(\psi) = \Ker(\phi_{A,\L})$ is $l$-divisible for any $l\in \L$. 
By Lem.\ \ref{lem:rhoXm}, $\vp_m =\rhoXm$ is injective for all $m\in \NL$: 
the condition (a) in Lem.\ \ref{lem:JS} holds. 


The tame fundamental group $\pitabX$ is defined by the wild inertia subgroups 
$I_{\kXx}^1$ for $x\in \Xinf$ in Def.\ \ref{def:piabXD} and \eqref{eq:tfg}. 
This group $I_{\kXx}^1$ is pro-$p$ so that 
we have 
$B_m = \piabX/m \isomto \pitabX/m$ for each $m\in \NL$. 
Taking the inverse limit,    
\begin{equation}
\label{eq:pipit}
\xymatrix@C=5mm{
B_{\L} = \piabX_{\L} \ar[r]^-{\simeq} & \pitabX_{\L}.
}
\end{equation}  
By local class field theory (of $k$) and the structure of the base field $k$ 
(e.g., \cite{Neu99}, Prop.\ 5.7 (ii)), 
$(G_k^{\ab})_{\L}$ is (topologically) finitely generated. 
For $(\pitabXg)_{\L}$ is finitely generated \eqref{eq:pitabXg},  
so is $B_{\L}$ by \eqref{eq:pipit}. 
%
Using the finiteness of the torsion part $(B_{\L})_{\tor}$, 
there exists $N\in \NL$ such that $N\cdot (B_{\L})_{\tor} = 0$: the condition (b) in Lem.\ \ref{lem:JS} holds. 
The claim (i) follows from Lem.\ \ref{lem:JS}. 
 
\sn
(ii) 
Putting 
$\phi^{\mathrm{t}} = \phi_{\pitabX,\L}$ \eqref{eq:whp}, 
the commutative diagram  
\[
 \xymatrix{
   \piabX \ar[d]_-{\phi }  \ar@{->>}[r]&  \pitabX 
   \ar[d]^-{\phi^{\mathrm{t}}}\\
   \piabX_{\L}  \ar[r]^{\simeq} & 
   \pitabX_{\L}  
 }
\]
induces a short exact sequence 
\[
\xymatrix@C=5mm{
  \displaystyle{\bigoplus_{x\in \Xinf} I_{\kXx}^{1}} \ar[r] & \Ker(\phi) \ar[r] &  
  \Ker(\phi^{\mathrm{t}}) \ar[r] & 0.
} 
\]
Recall that the wild inertia subgroup $I_{\kXx}^1$ is pro-$p$, 
in particular, $l$-torsion free. 
It is enough to show that $\Ker(\phi^{\mathrm{t}})$ is $l$-torsion free. 
We further consider the commutative diagram with exact rows: 
\[ 
  \xymatrix@C=10mm{
  0\ar[r] & \pitabXg\ar[d]^-{\phi^{\mathrm{t},0}} \ar[r] & \pitabX\ar[d]^-{\phi^{\mathrm{t}}} \ar[r] & 
  G_k^{\ab}\ar[d]^-{\phi_{k}} \ar[r] & 0\\
  0\ar[r] & (\pitabXg)_{\L} \ar[r] & \pitabX_{\L} \ar[r] &  (G_k^{\ab})_{\L} ,  
  }
\]
where  
$\phi^{\mathrm{t},0} := \phi_{\pitabXg,\L}$ 
and $\phi_k := \phi_{G_k^{\ab},\L}$ 
\eqref{eq:whp}.   
Since 
$\pitabXg$ is finitely generated \eqref{eq:pitabXg}, 
$\Ker(\phi^{\mathrm{t},0})$ is $l$-torsion free. 
By local class field theory, $\Ker(\phi_{k})$ is also $l$-torsion free.  
Therefore, the same holds on 
$\Ker(\phi^{\mathrm{t}})$. 
\end{proof}

\sn 
{\bf (Proof of Thm.\ \ref{thm:main2} - continued)}
By the exact sequence \eqref{eq:Kerdiv} and the claim above,  
$\Ker(\rhoX)$ is $l$-divisible for any prime $l\in\L$ as required. 
\end{proof}

\subsection*{Restricted Ramification}
\label{subsec:ray}
In closing this section, we derive 
the class field theory with \emph{modulus} 
from Thm.\ \ref{thm:main2} above. 


\begin{thm}
Let $D \ge 0$ 
be an effective Weil divisor on $\Xbar$ 
with support $|D| \subset  \Xinf$. 
For $\rho_{X,D}:\CXD \to \piabXD$, 
we have:

\begin{enumerate}
\item $\Ker(\rhoXD)$ is  
the maximal $l$-divisible subgroup of $\CXD$ for any prime $l\neq p$, and

\item $\Cokertop(\rhoXD) \simeq \Zhat^{\rXbar}$. 
\end{enumerate}
\end{thm}
\begin{proof}
The assertion (ii) is already given in \eqref{eq:cok(rhoXt)}. 
Furthermore, the surjective homomorphism $\CX \surj \CXD$ gives 
a homomorphism  
$\Ker(\rhoX) \surj \Ker(\rhoXD)$ which is also surjective. 
From Thm.\ \ref{thm:main2}, $\Ker(\rhoXD)$ is $l$-divisible for a prime $l\neq p$. 
Since profinite groups contain no non-trivial divisible elements, 
the assertion (i) follows. 
\end{proof}


\section{Proof of Thm.\ \ref{thm:mainintro}}
\label{sec:mainv}
We keep the notation of Section\ \ref{sec:main}. 

\subsection*{Unramified class field theory}

\begin{cor}
\label{cor:ucftv}
The induced homomorphism $\rhoXbarv : H^1(\Xbar,\QZ) \to  \CXbarv$ 
from the reciprocity map $\rhoXbar$ satisfies the following: 
\begin{enumerate}
\item  
$\Ker(\rhoXbarv) \simeq (\QZ)^{\rXbar}$, and 
\item $\rhoXbarv$ is surjective. 
\end{enumerate}
\end{cor}
\begin{proof}
	The assertion (i) follows from Thm.\ \ref{thm:ucft} (i). 
	By Thm.\ \ref{thm:ucft} (ii), $\rhoXbar^0$ defined in \eqref{eq:VX} induces an injection 
	$\rhoXbarm^0: \CXbar^0/m \inj \piabXbarg/m$. 
	Since the quotient $\CXbar^0/m$ is finite (Thm.\ \ref{thm:ucft} (iii)),  
	we obtain the 
	surjective homomorphism  
	\begin{equation}
	\label{eq:rhoXbarm0}
	\xymatrix@C=5mm{
		(\rhoXbarm^0)^{\vee}:(\piabXbarg/m)^{\vee} \ar@{->>}[r] & (\CXbar^0/m)^{\vee} 
	}\end{equation}
	on the dual groups for any $m\in \Z_{\ge 1}$. 
	Now, we show that 
	$\rhoXbargv : \piabXbargv \to  \CXbargv$  
	is surjective. 
	Take a character $\varphi \in \CXbargv$.
	By the very definition of  $\CXbargv$, 
		the character $\varphi$ has finite order (\Cf Notation).  
		Hence, there exists $m \in \Z_{\ge 1}$ 
		and $\varphi_m \in (\CXbarg/m)^{\vee}$ 
		such that $\varphi$ is the image of $\varphi_m$ 
		by the natural map $(\CXbarg/m)^{\vee}\to \CXbargv$. 
		Since $(\rhoXbarm^0)^{\vee}$ is surjective \eqref{eq:rhoXbarm0}, 
		there exists $\chi_m \in (\piabXbarg/m)^{\vee}$ such that 
		$(\rhoXbarm^0)^{\vee}(\chi_m) = \varphi_m$. 
		From the commutative diagram 
		\[
		  \xymatrix{
		     (\piabXbarg/m)^{\vee} \ar@{->>}[d]_{(\rho_{\Xbar,m}^0)^{\vee}} \ar[r] & 
		     \piabXbargv \ar[d]^{\rhoXbargv} \\
		    (\CXbarg/m)^{\vee} \ar[r] & (\CXbarg)^{\vee},
		  }
		\] 
		the image $\chi$  of $\chi_m$ by $(\piabXbarg/m)^{\vee} \to \piabXbargv$ 
		gives $\varphi= \rhoXbargv(\chi)$. 
	    Hence, $\rhoXbargv$ is surjective.
	    
	    On the other hand, the commutative diagram \eqref{eq:VX} 
	    and the Hochschild-Serre spectral sequence 
	    $\HGal^s(k, H^t(\Xbar_{\kbar},\QZ)) \Rightarrow H^{s+t}(\Xbar,\QZ)$
%
	    associated with the projection $\Xbar_{\kbar} \to \Xbar$  
	    	     (\Cf \cite{SGA4}, Exp.\ VIII, Cor.\ 8.5)  
	    give 
	    the following commutative diagram with exact rows: 
	    \begin{equation}
	    \label{eq:QZ}
	    \vcenter{
	     \xymatrix@C=4mm{
	     0\ar[r] & \HGal^1(k, \QZ) \ar[d]^{\rho_{k}^{\vee}}_{\simeq} \ar[r] & 
	    H^1(\Xbar, \QZ) \ar[d]^{\rho_{\Xbar}^{\vee}}\ar[r]  
	       & H^1(\Xbar_{\kbar},\QZ)^{G_k} \ar[d]  \ar[r]  & \HGal^2(k,\QZ)\\
	       & (k^{\times})^{\vee} \ar[r]^{N_{\Xbar}^{\vee}} & \CXbar^{\vee} \ar[r] &  
	    (\CXbarg)^{\vee}  &\qquad .
	     }
	    }
	    \end{equation}
	    Here, for a $G_k$-module $M$, we denote by $M^{G_k}$ the $G_k$-invariant 
	    submodule of $M$ and $\HGal^2(k,\QZ) = 0$ as in \eqref{eq:H2(k)}.  
	    By local class field theory, the left vertical map $\rho_{k}^{\vee}$ in \eqref{eq:QZ} is 
	    bijective. 
	    Because of 
	    $\HGal^1(k,\QZ) \simeq (G_k^{\ab})^{\vee}$, and 
	    $H^1(\Xbar,\QZ) \simeq \piabXbar^{\vee}$, 
%
	    we obtain $H^1(\Xbar_{\kbar},\QZ)^{G_k} \simeq (\piabXbarg)^{\vee}$. 
	    The right vertical map in the diagram \eqref{eq:QZ} coincides with $\rhoXbargv$ and is surjective by 
	    \eqref{eq:rhoXbarm0}. 
	    Therefore, $\rhoXbar^{\vee}$ is surjective. 
%
%
%
%
\end{proof}

%

\begin{cor}
	\label{cor:ucftmv}
	We assume that we have $r(\Xbar) = 0$. 
	Then  
	$\rhoXbarmv:H^1(\Xbar,\Z/m) \to (\CXbar/m)^{\vee}$ is bijective 
	for any $m\in \Z_{\ge 1}$. 		
\end{cor}
\begin{proof}
We have the following commutative diagram:
\[
\xymatrix{
H^1(\Xbar,\Z/m)  \ar[d]^-{\rho_{\Xbar,m}^{\vee}}\ar@{^{(}->}[r] & H^1(\Xbar,\QZ) \ar[d]^-{\rho_{\Xbar}^{\vee}}_{\simeq} \\ 
(\CXbar/m)^{\vee}\ar@{^{(}->}[r]  & \CXbarv,
}
\]
where the vertical maps are induced from  
$\piabXbar \surj \piabXbar/m$ and $\CXbar\surj \CXbar/m$. 
The assertion follows from Cor.\ \ref{cor:ucftv}.
\end{proof}

\subsection*{Open curves}
Recall that 
$X \subset \Xbar$  is a non-empty open subscheme 
and   
$\rhoX$ induces 
$\rhoXv : H^1(X,\QZ) \to \CXv$ and 
$\rhoXm^{\vee}: H^1(X,\Z/m) \to (\CX/m)^{\vee}$
for each $m\in \Z_{\ge 1}$.
\begin{prop}
	\label{prop:p}
Assume that $r(\Xbar) = 0$. 
Then  
$\rhoXm^{\vee}: H^1(X,\Z/m) \to (\CX/m)^{\vee}$ is bijective 
for any $m\in \Z_{\ge 1}$. 
\end{prop}
\begin{proof}
From the assumption $r(\Xbar) = 0$ 
and Thm.\ \ref{thm:coker}, 
$\rhoX$ and hence $\rhoXm$ has dense image. 
On the dual groups, $\rhoXmv$ is injective for any $m\in \Z_{\ge 1}$.  
In the following, we show that $\rhoXmv$ is surjective. 

\sn  
\textbf{(Prime to $p$-part)}  
For $m \in \Np$,  
we have an isomorphism 
$\piabX/m \simeq \pitabX/m$ of \emph{finite groups} as noted in the proof of Thm.\ \ref{thm:main2} 
(\Cf \eqref{eq:pipit}). 
Since 
$\rhoXm$ is an injective homomorphism of finite groups (Lem.\ \ref{lem:rhoXm}),  
the dual $\rhoXm^{\vee}$ becomes surjective.

\sn 
\textbf{($p$-part)} 
Instead of using $\Z/p^n$ with $\QZ$ in \eqref{eq:locx}, 
for each $x\in \Xinf$, 
we have the following commutative diagram with exact rows: 
\[
  \xymatrix@C=5mm{
 0 \ar[r] & \HGal^1(\kx,\Z/p^n) \ar[d]_{\simeq}^{\rho_{\kx,p^n}^{\vee}}\ar[r] & 
 \HGal^1(\kXx,\Z/p^n) 
 \ar[d]^{\rho_{\kXx,p^n}^{\vee}}_{\simeq} \ar[r]&  H^2_x(\Xbar,\Z/p^n) 
 \ar[r]\ar@{-->}[d]^{\phi_{x,p^n}} & 0 \\
 0 \ar[r] & (\kxt/p^n)^{\vee} \ar[r] & (K_2(\kXx)/p^n)^{\vee} 
 \ar[r]& (U^0K_2(\kXx)/p^n)^{\vee}  & & \hspace{-23mm},\\
 }
\]
where the middle vertical map is bijective \eqref{eq:rhoKmv}.
As in \eqref{eq:loc}, the localization 
sequence and \eqref{eq:rhoKmv} give the following diagram 
with exact rows:
\[\vcenter{
\entrymodifiers={!! <0pt, .8ex>+}
 \xymatrix@C=4mm@R=5mm{
 0\ar[r] & H^1(\Xbar, \Z/p^n) \ar[d]^{\rho_{\Xbar,p^n}^{\vee}}_{\simeq}\ar[r] 
 & 
 H^1(X, \Z/p^n) \ar[d]^{\rho_{X,p^n}^{\vee}}\ar[r]
   & \displaystyle{\bigoplus_{x\in \Xinf}H^2_{x}(\Xbar, \Z/p^n)} 
\ar@{^{(}->}[d]^-{\phi }  \ar[r]^-{j} & H^2(\Xbar,\Z/p^n)\\
  0 \ar[r] & (\CXbar/p^n)^{\vee} \ar[r] &  (\CX/p^n)^{\vee} \ar[r]^-{i} 
   & \displaystyle{\bigoplus_{x\in \Xinf}(U^0K_2(\kXx)/p^n)^{\vee}} & ,
 }}
\] 
where $\phi := \oplus \phi_{x,p^n}$. 
From Cor.\ \ref{cor:ucftmv}, 
$\rho_{\Xbar,p^n}^{\vee}$ is bijective. 
\begin{claim}
$\Im(i) \subset \Im(\phi)$. 
\end{claim}
\begin{proof}
The map $i$ can be written as the composition  
\[
\entrymodifiers={!! <0pt, .8ex>+}
\xymatrix@C=5mm{
(\CX/p^n)^{\vee} \ar[r] & \displaystyle{\bigoplus_{x\in \Xinf} (K_2(\kXx)/p^n)^{\vee}} \ar[r]^{\oplus i_x}& \displaystyle{\bigoplus_{x\in \Xinf} (U^0K_2(\kXx)/p^n)^{\vee}},
}
\] 
where the first map is given by the natural map $K_2(\kXx) \to \CX$ for each $x\in \Xinf$ and 
the latter which is denoted by $\oplus i_x$ 
is induced from the inclusion $U^0K_2(\kXx) \inj K_2(\kXx)$ for each $x\in X_{\infty}$. 
For each $x\in \Xinf$, as in \eqref{eq:locx}, 
there exists a commutative diagram
\[
\xymatrix{
\HGal^1(\kXx,\Z/p^n) \ar[d]_{\rho_{\kXx,p^n}^{\vee}}^{\simeq} \ar[r] & H^2_x(\Xbar,\Z/p^n) \ar[d]^{\phi_{x,p^n}} \\
(K_2(\kXx)/p^n)^{\vee} \ar[r]^{i_x} & (U^0K_2(\kXx)/p^n)^{\vee}.
}
\]
Here, the left vertical map is bijective \eqref{eq:rhoKmv} 
and the claim follows. 
\end{proof}

\sn
\textbf{(Proof of Prop.\ \ref{prop:p} - continued)} 
To show that $\rho_{X,p^n}^{\vee}$ is surjective, 
take $\vp\in (\CX/p^n)^{\vee}$. 
From the above Claim, 
there exists $\gamma \in \bigoplus_x H_x^2(\Xbar,\Z/p^n)$ such that 
$i(\vp) = \phi(\gamma)$. 
From $H^2(\Xbar \otimes_k \kbar,\Z/p^n) = 0$ 
(\cite{SGA4}, Exp.\ X, Cor.\ 5.2) 
and Cor.\ \ref{cor:ucftmv}, 
there exists a finite Galois extension $k'$ of $k$ such that 
the image of $j(\gamma)$ by 
the homomorphism 
$\fbar^{\ast}: H^2(\Xbar,\Z/p^n) \to H^2(\Xpbar,\Z/p^n)$ 
becomes zero, 
where 
$\fbar:\Xpbar := \Xbar \otimes_k k' \to \Xbar$ is the projection.  
%
Put also $X' := X\otimes_k k'$ and let $f:X'\to X$ be the induced morphism.  
In this setting, we have the norm homomorphism 
$N:= N_{X'/X}:\CXp \to \CX$ defined in Def. \ref{def:norm}.  
By \eqref{eq:norm}, this makes 
the following diagram commutative: 
\[
  \xymatrix{
    H^1(X,\Z/p^n) \ar[d]_{\rho_{X,p^n}^{\vee}}\ar[r]^{f^{\ast}} & H^1(X', 
\Z/p^n) \ar[d]^{\rho_{X',p^n}^{\vee}} \\
    (\CX/p^n)^{\vee} \ar[r]^{N^{\vee}_{p^n}} & (C(X')/p^n)^{\vee},
  }
\]
where $N^{\vee}_{p^n}$ is the induced homomorphism by $N = N_{X'/X}$.  
Thus, 
there exists  $\chi'\in H^1(X',\Z/p^n)$ such that   
$\vp' := N^{\vee}_{p^n}(\vp) = \rho_{X',p^n}^{\vee}(\chi')$ in 
$(C(\Xbar')/p^n)^{\vee}$ by the diagram chase.
It is left to show that  
$\vp$ comes from $H^1(X,\Z/p^n)$. 

Let $H$ be the $p$-Sylow subgroup of $G := \Gal(k'/k)$ 
and $k_H$ the fixed field of $H$ in $k'$. 
Putting $X_{H} := X\otimes_k k_H$, 
the diagram 
\[
  \xymatrix{
	  & H^1(X_{H},\Z/p^{n}) \ar[d]_{\rho_{X_H,p^n}^{\vee}} \ar[r] & H^1(X,\Z/p^n) 
	  \ar[d]^{\rho_{X,p^n}^{\vee}} \\ 
	  (\CX/p^n)^{\vee} \ar@/_3pc/[rr]^{[k_H:k]} 
	  \ar[r]^{N_{X_H/X,p^n}^{\vee}} 
	  & 
	  (C(X_H)/p^n)^{\vee} 
	   \ar[r]^{i_{X_H/X,p^n}^{\vee}} & (\CX/p^n)^{\vee}\\
	   &
  }
\]
is commutative by \eqref{eq:norm}. 
From Lem.\ \ref{lem:pf}, we have $N_{X_H/X}\circ i_{X_H/X} = 
[k_H:k]\Id_{\CX}$. 
Since the order of $\varphi$ is a power of $p$, 
using the above diagram, 
we may assume $k_H = k$ and $G = \Gal(k'/k)$ is a $p$-group.
Take a field extensions
$k = k_0 \subset k_1 \subset \cdots \subset k_s = k'$ such that  $k_{i+1}/k_i$ is a cyclic extension of degree $p$. 
By induction on $i$,   
we may assume that 
the Galois group $G$ is a cyclic group of the order $p$. 
We have the following commutative diagram:
\begin{equation}
\label{eq:G}
\vcenter{
  \xymatrix@C=5mm{
  0\ar[r] & G^{\vee} \ar[d]\ar[r]&  H^1(X,\Z/p^n) 
\ar[d]^{\rho_{X,p^n}^{\vee}}\ar[r]^-{f^{\ast}}& H^1(X',\Z/p^n)^{G} 
\ar[d]^{\rho_{X',p^n}^{\vee}} \ar[r] &  0 \\
  0\ar[r] & \Ker(N^{\vee}_{p^n}) \ar[r] & (\CX/p^n)^{\vee} \ar[r]^-{N^{\vee}_{p^n}} & 
(C(X')/p^n)^{\vee\ G} & .  
  }
  }
\end{equation}
The upper horizontal sequence is exact 
which comes from the Hochschild-Serre spectral sequence 
$\HGal^s(G, H^t(X',\Z/p^n)) \Rightarrow H^{s+t}(X,\Z/p^n)$
%
associated with $f:X'\to X$ (\Cf \cite{SGA4}, Exp.\ VIII, Cor.\ 8.5) 
and 
$H^2(G,\Z/p^n) = 0$ (\cite{NSW}, Prop.\ 6.5.10). 
Since $\vp'=N^{\vee}_{p^n}(\vp)$ is fixed by $G$ and 
$\rho_{X',p^n}^{\vee}$ is \emph{injective}, 
$\chi'$ is also fixed by $G$. 
From the diagram \eqref{eq:G}, 
it is enough to show that  
$\Ker(N_{p^n}^{\vee}) \simeq G^{\vee}$. 
Since the left vertical map in \eqref{eq:G} is injective, 
the lemma below implies that 
$\vp$ comes from an element of $H^1(X,\Z/p^n)$ 
and thus $\rho_{X,p^n}^{\vee}$ is surjective. 
\end{proof}
\begin{lem}
Let $k'/k$ be a Galois extension of $[k':k] = p$. 
We assume that the base change $X' := X\otimes_k k' \to X$ is not a completely split covering. 
Then, the following sequence is exact: 
\[
  \xymatrix@C=5mm{
  0\ar[r] & G^{\vee} \ar[r] & \CX^{\vee} \ar[r]^{N^{\vee}} & C(X')^{\vee},  
  }
\]
where  $G = \Gal(k'/k)$ and $N = N_{X'/X}$.
\end{lem}
\begin{proof}
  A character $\psi \in \Ker(N^{\vee})$ induces an element $\psi_x$  of $K_2(\kXx)^{\vee}$  
  for each $x \in \Xinf$. 
  Since $\psi_x$ is in the kernel of 
  $N_{k'\kXx/\kXx}^{\vee}:K_2(\kXx)^{\vee} \to K_2(k'\kXx)^{\vee}$, 
  the corresponding character $\chi_x := (\rho_{\kXx}^{\vee})^{-1}(\psi_x) \in 
  H^1(\kXx)$ (Thm.\ \ref{thm:lcft} (i))  
  is annihilated by the unramified extension $k'\kXx/\kXx$.  
  In particular, $\chi_x$ is unramified 
  so that 
  $\psi_x$ annihilates $U^0K_2(\kXx)$ (Thm.\ \ref{thm:lcft} (ii)).
  Thus, the assertion is reduced to the case of $X = \Xbar$, that is, 
  the exactness of 
  \[
  \xymatrix@C=5mm{
    0 \ar[r] & G^{\vee}\ar[r] &  \CXbar^{\vee} \ar[r]^{N_{\Xpbar/\Xbar}^{\vee}} & \CXpbar^{\vee}. 
  }
  \]  
  This follows from Cor.\ \ref{cor:ucftv}. 
\end{proof}
\begin{thm} 
\label{thm:main}
Suppose that we have $r(\Xbar) = 0$. 
Then, the dual of the reciprocity homomorphism 
$\rhoXv:\piabX^{\vee} \to \CX^{\vee}$ is bijective. 
\end{thm}
\begin{proof} 
We use 
$\piabX^{\vee} \simeq H^1(X,\QZ)$ \eqref{eq:H1pi}. 
The injectivity of $\rhoXv$ follows from Thm.\ \ref{thm:coker}.  
Take any $\varphi \in \CX^{\vee}$. 
Since $\varphi$ has finite order, 
$\varphi$ defines $\varphi_m \in (\CX/m)^{\vee}$ 
for some $m \in \Z_{\ge 1}$. 
	Consider the commutative diagram 
	\[
	  \xymatrix{
	    H^1(X,\Z/m) \ar[d]_{\rho_{X,m}^{\vee}}^{\simeq} \ar[r] & 
	    H^1(X,\Q/\Z) \ar[d]^{\rhoXv} \\
	    (\CX/m)^{\vee} \ar[r] & \CX^{\vee}
	  }
	\]
	where the left vertical map is bijective (Prop.\ \ref{prop:p}). 
	Therefore, one can find $\chi \in H^1(X,\QZ)$ such that 
	$\rhoXv(\chi) = \varphi$ and hence $\rhoXv$ is surjective.
\end{proof}

\bibliographystyle{amsplain}
\bibliography{cft_p}


\bigskip\noindent
Toshiro Hiranouchi \\
Department of Mathematics, Graduate School of Science, Hiroshima University
1-3-1 Kagamiyama, Higashi-Hiroshima, 739-8526 JAPAN\\
Email address: {\tt hira@hiroshima-u.ac.jp }

\end{document}